\documentclass[10pt]{amsart}
\usepackage{amssymb, amsmath}
\usepackage{graphics,xcolor}
\usepackage{latexsym,amsmath,amssymb,amsthm,amsfonts}
\usepackage{amscd}
\usepackage[arrow, matrix, curve]{xy}

\theoremstyle{plain}
\newtheorem{thm}{Theorem}[section]
\newtheorem{theorem}[thm]{Theorem}
\newtheorem{lemma}[thm]{Lemma}
\newtheorem{corollary}[thm]{Corollary}
\newtheorem{proposition}[thm]{Proposition}

\theoremstyle{definition}

\newtheorem{remark}[thm]{Remark}

\newtheorem{definition}[thm]{Definition}

\newtheorem{example}[thm]{Example}
\newtheorem{conjecture}[thm]{Conjecture}
\numberwithin{equation}{thm}
\newtheorem{defrem}[thm]{Definition/Remark}

\newcommand{\A}{{\mathbb A}}

\newcommand{\C}{{\mathbb C}}

\newcommand{\G}{{\mathbb G}}
\renewcommand{\H}{{\mathbb H}}

\newcommand{\N}{{\mathbb N}}

\newcommand{\Q}{{\mathbb Q}}
\newcommand{\R}{{\mathbb R}}
\newcommand{\SSS}{{\mathbb S}}

\newcommand{\U}{{\mathbb U}}
\newcommand{\V}{{\mathbb V}}
\newcommand{\W}{{\mathbb W}}

\newcommand{\rk}{{\rm rank}}

\title[]{Special subvarieties in Mumford-Tate varieties}
\author{Abolfazl Mohajer}
\author{Stefan M{\"u}ller-Stach}
\author{Kang Zuo}
\address{Universit\"{a}t Mainz, Fachbereich 08, Institut f\"ur Mathematik, 55099 Mainz, Germany}
\email{mueller-stach@uni-mainz.de, mohajer@uni-mainz.de, zuok@uni-mainz.de}
\thanks{Supported by SFB/TRR 45 der Deutschen Forschungsgemeinschaft}

\begin{document}

\begin{abstract}
Let $X=\Gamma \backslash D$ be a Mumford-Tate variety, i.e., a quotient of a Mumford-Tate domain $D=G(\R)/V$ by a discrete subgroup $\Gamma$. 
Mumford-Tate varieties are generalizations of Shimura varieties. 
We define the notion of a special subvariety $Y \subset X$ (of Shimura type), and formulate necessary criteria for $Y$ to be special. 
Our method consists in looking at finitely many compactified special curves $C_i$ in $Y$, and testing whether the inclusion $\bigcup_i C_i \subset Y$
satisfies certain properties. One of them is the so-called relative proportionality condition. In this paper, we give a new formulation of this numerical 
criterion in the case of Mumford-Tate varieties $X$. 
In this way, we give necessary and sufficient criteria for a subvariety $Y$ of $X$ to be a special subvariety of Shimura type in the sense of the Andr\'e-Oort conjecture. 
We discuss in detail the important case where $X=A_g$, the moduli space of principally polarized abelian varieties.
\end{abstract}

\maketitle

\section{Introduction} 

Griffiths domains \cite{cmp} are flag domains, i.e., quotients of the form $D=G(\R)/V$, where $G$ is a certain reductive algebraic group defined over $\mathbb{Q}$ and $V$ a compact stabilizer subgroup. 
Griffiths domains parametrize pure Hodge structures of given weight
and Hodge numbers. Any moduli space ${\mathcal M}$ of smooth, projective varieties induces, after a choice of cohomological degree and a base point, a period map 
\[
{\mathcal P}: {\mathcal M} \rightarrow \Gamma \backslash D,
\]
where $\Gamma$ is the monodromy group, i.e., the image of the fundamental group of ${\mathcal M}$ in $G(\R)$, 
a finitely generated, discrete subgroup. 

In general, the the period map ${\mathcal P}$ is not surjective, but has image contained in quotients 
of so-called Mumford-Tate domains by discrete subgroups, see \cite[Chap. 15]{cmp} or \cite{ggk}: 

\begin{theorem} \label{domaindecomp} After possibly replacing ${\mathcal M}$ by a finite, \'etale cover, the period map ${\mathcal P}$ factors as 
\[
{\mathcal P}: {\mathcal M} \longrightarrow \Gamma^{nc} \backslash D^{nc} \times \Gamma^{c} \backslash D^{c} \times D^{f},
\]
into a product of quotients of domains of non-compact, compact or flat (i.e., constant) type. Here $D^\bullet$ denotes a 
domain of the respective type. The composition with the third projection is constant. In addition, 
for each $x_1 \in \Gamma^{nc} \backslash D^{nc}$ and $x_3 \in D^{f}$, one has that 
${\rm Im}(\mathcal{P}) \cap (x_1 \times \Gamma^{c} \backslash D^{c}) \times x_3)$ is finite. 
\end{theorem}

Recall that a domain is by definition of non-compact type (resp. compact type, resp. flat type) if its
universal cover is a product of non-compact irreducible spaces (resp. product of compact irreducible spaces, resp. is a euclidean space).
Theorem~\ref{domaindecomp} asserts that the ''non-compact part'' of the period map is the essential one. 
The (derived) Mumford-Tate group of the Hodge structure of a general element in ${\mathcal M}$ contains the algebraic monodromy group,
i.e., the $\Q$-Zariski closure of the topological monodromy group, as a normal subgroup by a theorem of Y. Andr\'e \cite[Prop. 15.8.5]{cmp}. 

{\bf In the rest of this paper, we will assume that $G$ is of non-compact type, $\Q$-simple and adjoint.} It is not difficult to reduce to this case. 
Only in rare cases, $D$ itself is Hermitian symmetric \cite{cmp}. 
In these cases, $\Gamma \backslash D$ is a connected component of a Shimura variety under some arithmetic condition on $\Gamma$ \cite{deligne,moonen}. An important example 
is the moduli space $A_g=\Gamma \backslash \H_g$ of principally polarized abelian varieties of dimension $g$ with some level structure induced by $\Gamma$. Here
$\mathbb{H}_g$ denotes the Siegel upper half space of genus $g$. 
Shimura varieties contain distinguished subvarieties which are called special subvarieties. The zero-dimensional special subvarieties are the CM points, i.e., the 
points corresponding to Hodge structures with commutative Mumford-Tate group. Positive dimensional special subvarieties are more difficult to understand. 
However, the Andr\'e-Oort conjecture claims that special subvarieties of Shimura varieties are precisely the loci which are the Zariski closures of sets of CM points. 
This conjecture has recently attracted a lot of interest, see the work of Edixhoven, Klingler, Pila, Ullmo, Tsimerman, Yafaev and others 
\cite{edixhoven-yafaev,klingler-yafaev,pila,ullmo-yafaev,tsimerman}. In 2015, based on the above works, Tsimerman \cite{tsimerman} completed the proof of the Andr\'e-Oort conjecture for $A_g$ using 
an averaged version of a conjecture of Colmez proved by Yuan and Zhang \cite{zhang} and also independently by Andreatta-Goren-Howard-Madapusi in \cite{AGHM}.

Our aim is to give sufficient and effective Hodge theoretic criteria for a subvariety of $X=\Gamma \backslash D$ to be a special subvariety.  

In \cite{mvz09} and \cite{mz11}, we have studied special subvarieties in Shimura varieties of unitary or orthogonal type. 
Our method consisted of characterizing special subvarieties by a relative proportionality principle. The main goal of the present work is to generalize 
this principle to quotients of Mumford-Tate domains. 

\subsection{Results in the case $X=A_g$} \label{resultsAg}

For the reader's convenience, we first study the case where $X=A_g$. 
Let $A_g=\Gamma \backslash \H_g$ be a smooth model, i.e., we require that $\Gamma$ is torsion-free. Recall that
by $\mathbb{H}_g$ we denote the Siegel upper half space of genus $g$. 
We choose a smooth toroidal compactification $\overline{A}_g$ as constructed by Mumford et al. \cite[chap. III]{amrt}, 
such that the boundary $S \subset \overline A_g$ is a divisor with normal crossings. 

We consider a smooth projective subvariety $Y \subset \overline A_g$ meeting $S$ transversely and define $Y^0:=Y \cap A_g$.  Our results are valid
for any compactification $\overline{A}_g$ satisfying these conditions.
Throughout this paper we denote subvarieties contained in the locally symmetric part $A_g$ of $\overline A_g$ with a superscript $0$.

Such a subvariety $Y^0$ is called special, if it is an irreducible component of a Hecke translate of 
the image of some morphism $Sh_K(G,X)\rightarrow A_g=Sh_{K(N)}(GSp(2g),\H_g^{\pm})$, 
defined by an inclusion of a Shimura subdatum $(G,X)\subset(GSp(2g),\H_g^{\pm})$ 
together with some compact open subgroup $K\subset G(\mathbb{A}_f)$ such that $K\subset K(N)$. 
See Section~\ref{shimura} for details about Shimura varieties and special subvarieties.

We look for necessary and sufficient effective criteria, such that $Y^0$ is a special subvariety with minimal dimension
containing a union $\bigcup_{i \in I} C_i^0$ of finitely many special curves $C_i$. 
Already in our previous works \cite{mvz09} and \cite{mz11} we have found a necessary condition for $Y^0$ to be special, provided a 
compactified special curve $C \subset \overline{A}_g$ (i.e., compactification of a special curve $C^0\subset Y^0$) is contained in $Y$: 

\begin{definition}[Relative Proportionality Condition (RPC)]{~}\\
Let $C \subset Y \subset \overline{A}_g$ be a compactified irreducible special curve with logarithmic 
normal bundle $N_{C/Y}$, and $3$-step Harder-Narasimhan filtration 
$0 \subset N^0_{C/Y} \subset N^1_{C/Y} \subset N^2_{C/Y}=N_{C/Y}$ (both notions are explained in Section~\ref{relprop}). 
Then one has the relative proportionality inequality 
\begin{equation*} 
\deg N_{C/Y} \leq \frac{\rk (N^1_{C/Y})+\rk (N^0_{C/Y})}{2} \cdot \deg T_C(-\log S_C).
\end{equation*}
If $C$ and $Y$ are special subvarieties, then equality holds.  
\end{definition}

For curves $C$ on Hilbert modular surfaces or Picard modular surfaces, this condition is only a simple numerical criterion involving 
intersection numbers, see \cite{mvz09} and \cite{mz11}. 

Suppose we are given a finite number of compactified special curves $C_i$ in $\overline{A}_g$, 
contained in some irreducible subvariety $Y$ of dimension $\dim(Y) \ge 2$. We assume for simplicity that $Y$ and all $C_i$ 
intersect the boundary $S$ of $A_g$ transversely.
Fix a base point $y \in Y^0 \subset A_g$ contained in the union of all $C_i$ and assume for simplicity that the union $\bigcup_{i \in I} C_i^0$ is connected. 
Let $\mathbb{V}$ be the local system of weight 1 polarized $\mathbb{Q}$-Hodge structures over $A_g$. 
More precisely, choose a level $N\geq 3$ structure $A_g^{[N]}$ on $A_g$. Denote by $f: U \to A_g$ the universal family of abelian varieties, 
then $\V=R^1f_*\Q$ is the local system attached to it. So over each point $y \in A_g$, $\V_y$ is the associated polarized $\mathbb{Q}$-Hodge structure
$H^1(A_y,\Q)$ with $A_y$ being the abelian variety corresponding to $y$.
Let $(E,\theta)=(E^{1,0}\oplus E^{0,1}, \theta)$  be the Higgs bundle associated with $\V_{\C}$. The thickening of the Higgs field is
the pullback of the Higgs bundle on $Y$ via $\varphi:C\to Y$:
\[\theta_{C/Y}:= \varphi^{*}\theta:E^{1,0}\to E^{0,1}\otimes \varphi^{*} \Omega^1_Y(\log S_Y).\]
we can also consider the thickening in a point $y\in Y$:
\[\theta_{y\in Y}:E^{1,0}_y\to E^{0,1}_y\otimes \varphi^{*} \Omega^1_Y(\log S_Y)_y.\]
Consider the $k$-fold tensor product $(E,\theta)^{\otimes k}$. It decomposes as a direct sum
$E^{\otimes k}=\displaystyle \sum_{m+n=k} E^{m,n}$, where $E^{m,n}=\bigoplus E^{m_1,n_1}\otimes \cdots \otimes E^{m_k,n_k}$ and the sum ranges over all $(m_i, n_i)$ with $m_i+n_i=1, \displaystyle \sum_{i=1}^{k}m_i=m, \sum_{i=1}^{k}n_i=n$.
The Higgs field, which we continue to denote by $\theta$, decomposes as $\theta:E^{m,n}\to E^{m-1,n+1}\otimes \Omega^1_Y(\log S_Y)$
with 
\[\displaystyle E^{m_1,n_1}\otimes \cdots \otimes E^{m_k,n_k}\xrightarrow{\theta} \bigoplus_{i=1}^{k}E^{m_1,n_1}\otimes \cdots\otimes E^{m_i-1,n_i+1} \otimes\cdots E^{m_k,n_k}\otimes \Omega^1_Y(\log S_Y)\]
satisfying the Leibniz rule. 
In particular, given a base point $y\in \varphi(C) \subset Y$, we get the map
\[\theta_{y\in Y}:E^{p,p}_y\to E^{p-1,p+1}_y\otimes \Omega^1_Y(\log S_Y)_y.\] 
Note that the complex vector space $E^{p,p}_y$ does not have any natural $\Q$-structure, 
however since in the fibers $E^{p,p}_y\subset \V^{\otimes k}_y\otimes \C$, we define

\begin{definition} \label{Hy}
Let
\[W_{y \in Y}:=\{ v \in E^{p,p}_y\cap \V^{\otimes k}_{y,\Q} \mid \theta_{y \in Y}(v)=0\}.\]

i.e., the Hodge tensors that are killed by the (infinitesimal) Higgs field in $y\in Y$.  Note that if $k$ is not even, 
then we require that $W_{y \in Y}$ is trivial. The tensors in $W_{y \in Y}$ are called \emph{infinitesimally parallel Hodge tensors}. 
We define the group $H_y$ to be the largest $\mathbb{Q}$-algebraic group fixing the vectors in all $W_{y \in Y}$ ( for all $p$).
\end{definition}

The spaces $W_{y \in Y}$ (more generally $W_{y \in Y}$ for a finitely generated sublocal system $\mathbb{W}\subset \mathbb{V}^{\otimes}$, see Definition~\ref{kernels}) 
and the group $H_y$ will play a crucial role in the sequel, see \S 7.\\

For the following definition, note that as remarked above, by the construction of the Higgs bundle associated to a local system, it holds that in the fibers $E^{m,n}_y\subset \V^{\otimes k}_y$ and so
we can do the parallel transport using the connection fiberwise (but the connection associated to the local system does not descend
to the graded bundle $E^{m,n}$ in general). Note that it is not in general true that the property of being a $(p,p)$-class
is preserved under the parallel transport. 
\begin{defrem} \label{inclusions}
Let 
\[(E^{p,p}_y)_{par}=\{v\in E^{p,p}_y\cap \V^{\otimes k}_{y,\Q}|\text{any parallel transport of }v \text{ from } y \text{ to }  y^{\prime} \text{ lies in }
E^{p,p}_{y^{\prime}}\cap \V^{\otimes k}_{y^{\prime},\Q}\}.\]
Here $y^{\prime}\in Y$ varies in $Y$ and we say ``parallel transport of $v$ from $y$ to $y^{\prime}$'', where we mean more precisely 
"parallel transport of $v$ along any path from $y$ to $y^{\prime}$``. By parallel transporting one also gets $(E^{p,p}_y)_{par}\simeq (E^{p,p}_{y^{\prime}})_{par}$.

We define
\[MT(\mathbb{V})_y=\{g\in GL(E_y)|g \text{ stabilizes } v\in (E_y^{p,p})_{par} \text{ for all } p\in \N\}.\]
Note that $(E^{p,p}_y)_{par}$ is associated to a polarized sublocal system and since it consists of Hodge tensors (i.e., $\Q$-tensors of type $(p,p)$),
its polarization is definite and so it is unitary and hence $v\in (E^{p,p}_y)_{par}$ is killed by the Higgs field $\theta_{y \in Y}$. Hence, by the above defintion of the group $H_y$, we have the inclusion
\[H_y\subset MT(\mathbb{V})_y.\]
On the other hand, if $MT(\mathbb{V}_y)$ is the Mumford-Tate group fixing all Hodge tensors in $\V_y$, then it trivially holds that $MT(\mathbb{V}_y)\subset H_y$. 
So in general we have the inclusions
\[MT(\mathbb{V}_y)\subset H_y\subset MT(\mathbb{V})_y.\]
Note that by parallel transport we have an isomorphism $MT(\mathbb{V})_y\simeq MT(\mathbb{V})_{y^{\prime}}$ (but in general the isomorphism $MT(\mathbb{V}_y)\simeq MT(\mathbb{V}_{y^{\prime}})$
is not true).   
\end{defrem}

Now we are ready to state our first result:  

\begin{theorem}\label{Theorem1} Let $Y^0$ be a smooth, algebraic subvariety of $A_g$ such that
$Y^0$ has unipotent monodromies at infinity. Let $Y$ be a smooth compactification of $Y^0$ as above such that $Y$ intersects the boundary $S$ of $\overline{A}_g$ transversely.
Suppose $Y^0$ contains finitely many special curves $C^{0}_i$ such that the compactification $C_i$ of $C^{0}_i$ is a special curve in $Y$ and 
that $\bigcup_{i \in I} C_i^0$ is connected. Choose a base point $y\in \bigcup_{i \in I} C_i^0$. 
Assume the following:  \\ 
(BIG) The $\Q$-Zariski closure in $G={\rm Sp}(2g)$ of the monodromy representation of $\pi_1(\bigcup_{i \in I} C_i^0,y)$ equals the $\Q$-Zariski closure of the 
representation of $\pi_1(Y^0,y)$. \\
(LIE) If $H=H_y$ is the $\mathbb{Q}$-algebraic group in Definition~\ref{Hy} and $H(\mathbb{R})/K$ is the associated period domain, then one has $\dim H(\mathbb{R})/K \le \dim Y$. \\
(RPC) All compactified special curves $C_i$ satisfy relative proportionality.\\
Then, $Y^0$ is a special subvariety of $A_g$. 
\end{theorem}

In addition, the proof of the theorem implies that $K$ is a maximal compact stabilizer group, that $H$ is of Hermitian type (i.e., $H/K$ is a Hermitian symmetric space), and in the Hodge decomposition 
$\mathfrak{h}_\C=\mathfrak{h}^{-1,1} \oplus \mathfrak{h}^{0,0} \oplus \mathfrak{h}^{1,-1}$ of the real Lie algebra $\mathfrak{h}={\rm Lie} \, H(\R)$, one has 
$\mathfrak{h}^{-1,1}=T_{Y^0,y}$ for the holomorphic tangent space of $Y^0$ at $y$. In particular, the group $H_y$ does not depend on the base point $y$ in a crucial way.

More can be said about the group $H=H_y:$  In fact, in general the holomorphic tangent space
$T_{Y^0,y}$ is a subspace of the holomorphic tangent space of $H/K$ at $y$ and  
$H(\R)/K \subset \H_g$ is the smallest Mumford-Tate subdomain which contains $y$ and such that the holomorphic tangent space of $H(\R)/K$ at $y$ contains 
$T_{Y^0,y}$, see Proposition~\ref{Mon+Lie}.  Therefore, one always has $\dim H(\R)/K\geq \dim Y^0$, and condition (LIE) implies that $\dim H(\R)/K=\dim Y^0$. 

Theorem~\ref{Theorem1} generalizes previous work in \cite{mvz09} and \cite{mz11}, which was restricted to special subvarieties in
unitary or orthogonal Shimura varieties, hence the case of rank $\le 2$.
There are explicit examples of connected cycles $\cup_i C_i^0$ of special curves $C_i^0$ in $A_g$ for $g \ge 2$, 
for which the minimal enveloping special subvariety of $\cup_i C_i^0$ is $A_g$ but not smaller. For example two so called Mumford curves (see \cite{VZ04}, section 5) in $A_4$ intersecting at one point.
This shows that condition (LIE) is necessary. 
We saw already above that (RPC) is also necessary. Condition (BIG) is probably not a necessary condition. 
All three conditions are not independent, but their relations are not fully understood. 
In the course of the proof, we will see that condition (BIG) together with (RPC) implies that the monodromy group $\Gamma$ (i.e. the $\Q$-Zariski closure $\overline{\rho (\pi_1(Y^0,y))}$)
is contained in the group $H=H_y$ defined above. We therefore formulate a condition as follows 
\[
\textrm{(Mon)} \quad \quad \Gamma \subset H_y
\]

See the last section of this introduction for a strategy of the proof of Theorem~\ref{Theorem1}.

\subsection{Results in the case of a Mumford-Tate variety $X=\Gamma \backslash D$} \label{generalresults}

Now we turn to the general case. 
As far as we know, there is no good notion of Hecke operators on Mumford-Tate domains $D=G(\R)/V$. In addition, 
there are no good compactifications of a Mumford-Tate variety $X=\Gamma \backslash D$  known in these cases in general \cite{grt}.  

Therefore, to avoid these two difficulties, by a special curve in $X$ we will denote an \'etale morphism
\[
\varphi^0: C^0 \longrightarrow X                                                                    
\]
from a Shimura curve $C^0$, which is induced from a morphism of algebraic groups $G' \to G$ defined over $\Q$, such that a certain Shimura datum for $G'$ defines $C^0$.
Assume also that we are given a quasi-projective variety $Y^0 \subset X$ containing the image of $\varphi^0$ and with a NC smooth compactification $Y$ and let $C$ be a compactification of $C^0$.
Denote by $S_Y=Y \setminus Y^0$ the boundary divisor, and by $S_C=C \setminus C^0$, so that $S_C$ is the pullback of $S_Y$ to $C$, and 
$\varphi^0$ extends to a finite map $\varphi: C \to Y$.

In Section~\ref{relprop2}, we show that there is a filtration
\[
N^{0}_{C/Y}\subset N^{1}_{C/Y}\subset \cdots \subset
N^{s}_{C/Y}=N_{C/Y}
\]
on the logarithmic normal bundle $N_{C/Y}$, induced by the Harder-Narasimhan filtration on $N_{C/X}$. The logarithmic normal bundle
$N_{C/Y}$ is defined by the exact sequence
\[
0 \to T_C(- \log S_C) \to \varphi^* T_Y(- \log S_Y) \to N_{C/Y} \to 0.
\]
The relative proportionality condition can be stated as: 

\begin{definition}[Relative Proportionality Condition (RPC)]{~}\\
The curve $\varphi:C\to Y$ satisfies the
\emph{relative proportionality condition} (RPC) if the slope inequalities
\[
\mu(N^{i}_{C/Y}/ N^{i-1}_{C/Y})\leq \mu(N^{i}_{C/X}/N^{i-1}_{C/X}), \text{ for  } i=0,...,s
\]
are equalities. The sheaves $N^{i}_{C/X}$ are properly defined in Section~\ref{relprop2}. 
The integer $s$ depends on $C$ and $X$.
Summing up these inequalities, yield the relative proportionality inequality 
\begin{equation*} 
\deg N_{C/Y} \leq r(C,Y,X) \cdot \deg T_C(-\log S_C),
\end{equation*}
where $r(C,Y,X) \in \Q$ is a rational number depending on $C$, $Y$ and $X$, and hence on $G$.
If $C$ and $Y$ are special subvarieties, then equality holds. 
\end{definition}
Let $\mathfrak{g}={\rm Lie}(G)$ be the Lie algebra of $G$. We have a weight zero Hodge structure on ${\mathfrak g}$, 
\[
{\mathfrak g}=\bigoplus_p {\mathfrak g}^{-p,p}. 
\] 

\begin{definition} \label {horizontal} We denote by $T^h_X$ the holomorphic, horizontal tangent bundle to $X$ \cite[Sec. 12.5]{cmp}. That is, 
$T^h_X$ is the homogenous bundle on $X$ associated to ${\mathfrak g}^{-1,1}$. 
\end{definition}
Now we prove the analogue of Theorem~\ref{Theorem1} for Mumford-Tate varieties. We will assume that 
$Y^0$ is a \emph{horizontal subvariety} of $X$, i.e., that $T_{Y^0}$ is contained in the horizontal tangent bundle $T^h_X$. For the following theorem, we 
need also to generalize the notion of special subvariety of Shimura type of a general Mumford-Tate variety $X$. This will be done in section 5. 
Let $\mathbb{V}$ be a local system of polarized $\Q$-Hodge structures over $X$. So over each point $y \in X$, $\V_y$ is a polarized $\Q$-Hodge structure 
of some given weight. Choose any finitely generated, sublocal system $\W \subset \V^\otimes$ of even weight $2p$ and defined over $\Q$, 
where $\V^\otimes$ is the full tensor algebra generated by tensor powers of $\V$ and its dual. 
We denote the fiber of $\W_\Q$ over $y$ by $W_{y,\Q}$.  
Let $(E,\vartheta)$ be the Higgs bundle corresponding to $\W_\C$.

Assume now that $\varphi: C \to Y$ compactifies $C^0 \longrightarrow Y^0 \longrightarrow X$ with $C^0$ a special curve in $X$.  
If $\varphi:C\to Y$ satisfies (RPC), then we have a decomposition:
\[
N_{C/Y}=N^{0}_{C/Y}\oplus N^{1}_{C/Y}/N^{0}_{C/Y}\oplus \cdots \oplus
N^{i}_{C/Y}/N^{i-1}_{C/Y}\oplus \cdots \oplus
N^{s}_{C/Y}/N^{s-1}_{C/Y}.
\]

Now let 
\[
\theta_{y \in Y}:= E_y^{p,p} \to E_y^{p-1,p+1} \otimes \Omega^1_Y(\log S_Y)|_y 
\]
be the thickening of the Higgs field along $C$, as explained in section~\ref{generalresults}, with splitting
\[
E_y^{p-1,p+1} \otimes \Omega^1_Y(\log S_Y)|_y  \cong E_y^{p-1,p+1} \otimes 
\left(\Omega^1_C(\log S_C)|_y \oplus N_{C/Y}^\vee|_y \right).
\]
Although the complex vector space $E^{p,p}_y$ does not have any natural $\Q$-structure, we define as in the Definition~\ref{Hy}, 

\begin{definition} \label{kernels} Under these assumptions, we define a complex vector space in analogy with Definition~\ref{Hy}
\[
W_{y \in Y}:= \{ t \in E_y^{p,p}\cap \W_{y,\Q} \mid \theta_{y \in Y}(t)=0 \}.
\]
As in the case of $A_g$, the tensors in $W_{y \in Y}$ are called infinitesimally parallel Hodge tensors. 
We define the group $H_y$ to be the largest $\mathbb{Q}$-algebraic group fixing the vectors in all $W_{y \in Y}$.
\end{definition}

With the above definition, the condition (Mon) in the general case is fomulated exactly as in the previous section, namely the monodromy group $\Gamma$ 
(i.e. the $\Q$-Zariski closure $\overline{\rho (\pi_1(Y^0,y))}$) is contained in the group $H=H_y$ defined above.
\[
\textrm{(Mon)} \quad \quad \Gamma \subset H_y
\]

\begin{theorem}\label{Theorem2} Let $X=\Gamma \backslash D$ be a Mumford-Tate variety associated 
to the Mumford-Tate group $G$. Let $Y^0$ be a smooth, horizontal algebraic subvariety of $X$ that
 has unipotent monodromies at infinity. Moreover let $Y$ be a NC smooth compactification of $Y^0$. Suppose that there exists a finite collection of special curves $\varphi_i^0: C_i^0 \longrightarrow Y^0$
  such that $\bigcup_{i \in I} C_i^0$ is connected with compactifications $\varphi_i: C_i \longrightarrow Y$ as above and NC divisors $S_{C_i}=C_i\setminus C_i^0$. 
  Choose a base point $y\in \bigcup_{i \in I} C_i^0$. Assume the following: \\
(BIG) The $\Q$-Zariski closure in the Mumford-Tate group $G$ of the monodromy representation of 
$\pi_1(\bigcup_{i \in I} C_i^0,y)$ equals the $\Q$-Zariski closure of the representation of $\pi_1(Y^0,y)$. \\
(LIE) If $H=H_y$ is the $\Q$-algebraic group of Definition~\ref{kernels} and $K$ a compact subgroup such that $H(\R)/K \subset D$ is the period domain associated to $H$, then one has $\dim H(\R)/K \le \dim Y$.\\
(RPC) All compactified special curves $C_i$ satisfy relative proportionality.\\
Then, $Y^0$ is a special subvariety of $X$ of Shimura type (see Definition~\ref{specialdef2}).   
\end{theorem}

In addition, as in the case of $A_g$, the proof of the theorem implies that the group $H$ essentially does not depend on $y$, 
$K$ is a maximal compact stabilizer group, that $H$ is of Hermitian type, 
and in the Hodge decomposition $\mathfrak{h}_\C=\mathfrak{h}^{-1,1} \oplus \mathfrak{h}^{0,0} \oplus \mathfrak{h}^{1,-1}$ 
of the real Lie algebra $\mathfrak{h}={\rm Lie} \, H(\R)$, one has 
$\mathfrak{h}^{-1,1}=T_{Y^0,y}$ for the holomorphic tangent space of $Y^0$ at $y$. See Proposition~\ref{Mon+Lie} and the paragraph following this proposition for the above claims. 
 
The (RPC) condition implies that the above filtration is in fact the Harder-Narasimhan filtration on $N_{C/Y}$. 

\subsection*{Strategy of the proof}

The proof of both theorems is based on the following observations: 

\begin{proposition} \label{Mon+Lie} Let $X=\Gamma \backslash D$ be a Mumford-Tate variety associated 
to the Mumford-Tate group $G$. Let $Y^0$ be a smooth, horizontal algebraic subvariety of $X$ such that
$Y^0$ has unipotent monodromies at infinity. Assume the conditions (Mon) and (LIE). Then $Y^0$ is special. 
\end{proposition}

\begin{proof} Let $\Gamma$ be the image of $\pi_1(Y^0,y)$ under $\rho$ in $G$. The map of universal covers $\widetilde{Y^0}\to G(\R)/V$ factors through $\widetilde{Y^0}\to H(\R)^{+}/K\hookrightarrow G(\R)/V$
and using condition (Mon), the period map may be viewed as a map
\[
Y^0 {\buildrel \mathcal{P} \over \hookrightarrow} Z^0 =\Gamma \backslash H(\R)^+/K.
\]
By condition (LIE), one has $Y^0=Z^0$ for dimension reasons. 
Since $Y^0$ is horizontal by assumption, it follows that $Y^0$ is special. 
\end{proof}

Note that the tangent space of the image of the above period map $\mathcal{P}$ lies in the horizontal subspace of the tangent space of $H(\mathbb{R})/K$.
By the condition (LIE), the tangent space of the image of the period map is equal to the tangent space of $H(\mathbb{R})/K$. Hence the tangent space of $H(\mathbb{R})/K$ itself is horizontal.
Therefore $H$ is of Hermitian type and in particular $K$ is maximal compact.

Using this Proposition, the proofs of Theorem~\ref{Theorem1} and Theorem~\ref{Theorem2} are reduced to the proof of the 
following Theorem: 

\begin{theorem}\label{Theorem3} Let $X=\Gamma \backslash D$ be a Mumford-Tate variety associated 
to the Mumford-Tate group $G$. Let $Y^0$ be a smooth, horizontal algebraic subvariety of $X$ such that
$Y^0$ has unipotent monodromies at infinity. Then, conditions (BIG) and (RPC) imply condition (Mon). 
\end{theorem}

The condition (BIG) may be replaced by other conditions: for example, one may require that there is an integral linear combination 
$\sum_{i \in I} a_iC_i$ which deforms in $X$ and fills $X$ out. In \cite{mz11}, we showed that this assumption implies condition (BIG) as well.
In this light, we pose the following 

\begin{conjecture} \label{coverspecial}
Suppose an irreducible Mumford-Tate variety $X$ associated to $G$ contains (infinitely many) special curves. Then, condition (BIG) holds for $X$, i.e., 
there are finitely many compactified special curves $C_i$ in $X$, such that the $\Q$-Zariski closure of the monodromy representation of 
$\pi_1(\bigcup_{i \in I} C_i^0,y)$ is equal to $G$. 
\end{conjecture} 

In addition, one wants to find an effective bound of the number of special curves needed.
Conjecture~\ref{coverspecial} is known to be true in the case where $G=SO(2,n)$ and $G=SU(1,n)$ for $n \ge 1$, see \cite[Remark 3.7]{mz11}.
However, it appears to be open even in the case $G={\rm Sp}_{2g}$ for large $g$. 
It may be possible to solve this conjecture by looking at one special curve $C^0$ containing a CM-point $y$ and taking 
finitely many Hecke translates of $C^0$ which fix the point $y$.

Theorem~\ref{Theorem3} will be proved in the last section. In the sections before, we recall the notions of special subvarieties and 
explain the condition (RPC). 

\subsection*{Acknowledgements} This work and the first author were supported by a project of M\"uller-Stach and Zuo 
funded inside SFB/TRR 45 of Deutsche Forschungsgemeinschaft. We would like to thank C. Daw, B. Edixhoven and E. Ullmo for discussion and the 
referee for helpful comments and suggestions. 

\section{Mumford-Tate groups and Hodge classes} \label{mumtate}

For any $\Q$-algebraic group $M$, we denote by $M_\R$ the associated $\R$-algebraic group.
Let $V$ be a $\Q$-Hodge structure with underlying $\Q$-vector space also denoted by $V$. This corresponds to a real representation 
\[
h: \SSS \longrightarrow GL(V_{\R}) 
\]
of the Deligne torus $\SSS={\rm Res}_{\C/\R} \G_m$. 

\begin{definition} 
The (large) Mumford-Tate group $MT(V)$ of $V$ is the smallest $\Q$-algebraic subgroup of $GL(V)$ such that $MT(V_{\R})$ contains the image of $h$.
The (special) Mumford-Tate group, or Hodge group, $Hg(V)=SMT(V)$ is the smallest $\Q$-algebraic subgroup of $SL(V)$ such that $SMT(V_{\R})$ contains the image 
of the subgroup ${\rm S^1} \subset \SSS$. 
\end{definition}

Depending on the context, we will use both groups under the general name Mumford-Tate group.
If one looks at all Hodge classes in $V^{\otimes i} \otimes V^{\vee \otimes j}$ for all $(i,j)$, then the special Mumford-Tate group $SMT(V)$
is precisely the largest $\Q$-algebraic subgroup $G \subset Sp(2g)$ fixing all Hodge classes in such tensor products. 

\begin{example} Let us look at Hodge structures of weight $1$. 
We fix a level $N$ structure $A_g^{[N]}$ on $A_g$ with 
$N \ge 3$. Therefore, there is a universal family $f \colon U \to A_g$ over $A_g$. 
Let $\V=R^1f_*\Q$ be the natural $\Q$-local system of weight one on $A_g$. We denote by 
\[
\V_{\C}^\otimes= \bigoplus_{i,j} \V_{\C}^{\otimes i} \otimes \V_{\C}^{\vee \otimes j} 
\]
the full tensor algebra. This is an infinite direct sum of polarized local systems, where each summand $\V_{\C}^{\otimes i} \otimes \V_{\C}^{\vee \otimes j}$ 
carries a family of Hodge structures of weight $i-j$. A Hodge class in $\V^\otimes_{\C}$ is a flat section in some finite dimensional subsystem 
of $\V^\otimes_{\C}$ defined over $\Q$ and corresponding fiberwise to a $(p,p)$-class.
\end{example} 

\section{Special Subvarieties in $A_g$} \label{shimura}

Let us recall some useful notation concerning Shimura varieties and their special subvarieties. 

\begin{definition}[Shimura datum]
A Shimura datum is a pair $(G,X)$ consisting of a connected, reductive algebraic group
$G$ defined over $\Q$ and a $G(\R)$-conjugacy class $X \subset {\rm Hom}(\SSS,G_\R)$
such that for all (i.e., for some) $h \in X$, \\
(i) The Hodge structure on ${\rm Lie}(G)$ defined by ${\rm Ad} \circ h$ is of type 
$(-1,1)+(0,0)+(1,-1)$. \\
(ii) The involution ${\rm Inn}(h(i))$ is a Cartan involution of $G^{\rm ad}_\R$. \\
(iii) The adjoint group $G^{\rm ad}$ does not have factors defined over $\Q$ onto which $h$ 
has a trivial projection. 
\end{definition}

The connected components of $X$ form $G(\R)^+$-conjugacy classes. 
The weight cocharacter $h \circ w: \G_{m,\C} \to G_\C$ does not depend on the choice of $h$. \\

For a compact open subgroup $K\subset G(\mathbb{A}_f)$,  we consider the double coset space $Sh_K(G,X) =G(\mathbb{Q})\setminus (X\times G(\mathbb{A}_f)/K)$.
We denote an element of this set by $[x,aK]$ ($x\in X$, $aK\in G(\mathbb{A}_f)/K$). The Shimura variety associated to the above Shimura datum
is defined as $Sh(G,X)=\varprojlim Sh_K(G,X)$. Note that $Sh_K(G,X)$ can be recovered from $Sh(G,X)$ as the quotient modulo $K$ by the theory of Shimura varieties. 
An element $\gamma \in G(\mathbb{A}_f)$ defines a map $Sh_{K^{\prime}}(G,X)\to Sh_K(G,X)$,  $[x,aK^{\prime}]\mapsto [x,a\gamma K]$ called the 
\emph{Hecke translate} for $K^{\prime}\subset \gamma K \gamma^{-1}$. This map induces a (right) action of $G(\mathbb{A}_f)$ on $Sh(G,X)$.
We refer to \cite{moonen} for an accessible reference concerning Shimura varieties. \\

A morphism of Shimura data $(M, X^{\prime})\to (G,X)$ is a homomorphism $M\to G$ of algebraic groups
sending $X^{\prime}$ to $X$. A morphism of Shimura varieties $Sh(M, X^{\prime})\to Sh(G,X)$ is an inverse system of
regular maps of algebraic varieties compatible with the action of $G(\mathbb{A}_f)$.\\

Denote by $(GSp(2g),\H_g^{\pm})$ the Shimura datum in the sense of Deligne \cite{deligne} 
defining $A_g=A_g^{[N]}$ with level structure given by the compact open subgroup $K(N)$ of $GSp(2g)(\mathbb{A}_f)$.

\begin{definition}[Special Subvarieties]\label{specialdef}
A special subvariety of $A_g$ is  a geometrically irreducible component of a Hecke translate of 
the image of some morphism $Sh_K(G,X)\rightarrow A_g=Sh_{K(N)}(GSp(2g),\H_g^{\pm})$, 
which is defined by an inclusion of a Shimura subdatum $(G,X)\subset(GSp(2g),\H_g^{\pm})$ 
together with some compact open subgroup $K\subset G(\mathbb{A}_f)$ such that $K\subset K(N)$.
\end{definition}

In other words, by the above notation, there is a sequence 
\[
Sh(G,X)_\C \longrightarrow Sh(GSp(2g),\H_g^{\pm})_\C {\buildrel \gamma \over \longrightarrow} 
Sh(GSp(2g),\H_g^{\pm})_\C {\buildrel \text{quot} \over \longrightarrow} A_g=Sh_{K(N)}(GSp(2g),\H_g^{\pm})
\]
where $\gamma \in G(\A_f)$ and the second map is the Hecke operator defined above.\\ 

Special subvarieties are totally geodesic subvarieties with respect to the natural Riemannian (Hodge) metric, 
i.e., geodesics which are tangent to a special subvariety stay inside. 
In fact, there is almost an equivalence by a result of Abdulali \cite{Abdul} and Moonen: 

\begin{proposition}
An irreducible algebraic subvariety of $A_g$ is special if and only if it is totally geodesic and contains a CM point. 
\end{proposition}

\begin{proof} See Theorem 6.9.1 in Moonen \cite{moonen}. 
\end{proof}

\section{Relative Proportionality in $A_g$} \label{relprop}

Consider a non-singular projective curve $C$ and an embedding
$$
\varphi: C \hookrightarrow Y \hookrightarrow \overline{A}_g,
$$
where $Y \subset \overline{A}_g$ is a smooth projective subvariety as in the introduction. We denote by 
$C^0:=\varphi^{-1}(Y^0)\not=\emptyset$ the ''open'' part, where $Y^0=Y \cap A_g$. Assume that $C^0$ is a special curve in the following. 
Let $S_C$ and $S_Y$ be the intersections of $C$ and $Y$ with $S$. We assume overall that such intersections are transversal. \\

The logarithmic normal bundles of $C$ in $Y$ and $\overline{A}_g$ are defined by the exact sequences
\[
0 \to T_C(-\log S_C)  \to T_{\overline{A}_g}(-\log S) \to N_{C/\overline{A}_g} \to 0,
\]
\[
0 \to T_C(-\log S_C)  \to T_{Y}(-\log S_Y) \to N_{C/Y} \to 0.
\]

Let $N^\bullet_{C/Y}$ be the Harder-Narasimhan filtration on the logarithmic normal bundle $N_{C/A_g}$ intersected with $N_{C/Y}$.
The following definition was given in \cite[Def. 1.4]{mz11}.

\begin{definition}[Relative Proportionality Condition (RPC)]{~}\\ 
The map $\varphi: C \hookrightarrow Y$ satisfies the relative proportionality condition (RPC), if the slope inequalities
\[
\mu(N_{C/Y}^{i}/N_{C/Y}^{i-1})\leq\mu (N_{C/\overline A_g}^{i}/N_{C/\overline A_g}^{i-1}),\quad i=0,1,2
\]
are equalities. For the slopes, one gets by \cite{mz11}:
\begin{align*}
\mu (N_{C/\overline A_g}^{2}/N_{C/\overline A_g}^1)    & =0, \cr 
\mu (N_{C/\overline A_g}^{1}/N_{C/\overline A_g}^0)    & =\frac{1}{2} \deg T_C(-\log S_C), \cr 
\mu (N_{C/\overline A_g}^{0})                          & =\deg T_C(-\log S_C).
\end{align*}
Hence, we obtain a set of inequalities 
\begin{eqnarray*}
\mu(N_{C/Y}^{2}/N_{C/Y}^1) & \le & 0, \\
\mu(N_{C/Y}^{1}/N_{C/Y}^0) & \le & \frac{1}{2} \deg T_C(-\log S_C), \\  
\mu(N_{C/Y}^{0})           & \le & \deg T_C(-\log S_C). 
\end{eqnarray*}
Adding all three inequalities we obtain a single inequality
\begin{equation*} \label{rpc}
\deg N_{C/Y} \leq \frac{\rk (N^1_{C/Y})+\rk (N^0_{C/Y})}{2} \cdot \deg T_C(-\log S_C).
\end{equation*}
In case of equality, we say that (RPC) holds.
\end{definition}

\begin{example}\label{surfaces} 
In case $Y$ is a smooth projective surface, and $C$ is a smooth special curve in $Y$ intersecting the boundary $S_Y$ 
transversally, then 
\[
(K_Y+S_Y)\cdotp C+2C^2=0,
\]
if $Y$ is a Hilbert modular surface, and
\[
(K_Y+S_Y)\cdotp C+3C^2=0,
\]
if $Y$ is a ball quotient, see \cite[Thm. 0.1]{mvz09}, \cite[Ex. 1.6]{mz11} and \cite[Chap. 17]{cmp}.
\end{example} 

The main consequence of (RPC) is the following: 

\begin{proposition} \label{splitting}  {~} \\
(i) If $\varphi: C \hookrightarrow  Y$ satisfies (RPC), then $\varphi^*T_Y(-\log S_Y)$
is a direct summand of an orthogonal decomposition of $\varphi^*T_{\overline A_g}(-\log S)$ with respect to the Hodge metric.\\
(ii) If $Y^0 \hookrightarrow A_g$ is a special subvariety, then $\varphi^*T_Y(-\log S_Y)$
is a direct summand of an orthogonal decomposition of $\varphi^*T_{\overline A_g}(-\log S)$ with respect to the Hodge metric and
$\varphi: C \hookrightarrow Y$ satisfies (RPC).
\end{proposition}

\begin{proof}  \cite[Prop. 1.5]{mz11}.   
\end{proof}

In \cite[Formula 1.3]{mz11} we showed that, if $C^0$ is a special curve, one has a splitting
\[
\varphi^*  T_Y(-\log S_Y) \cong T_C(-\log S_C) \oplus N_{C/Y}.
\]
This splitting is induced from a corresponding splitting of $\varphi^*T_{\overline A_g}(-\log S)$. 
If, in addition, (RPC) holds, then this splitting is compatible with the decomposition
\[
N_{C/Y}=\bigoplus_{i=0}^2 N^{i}_{C/Y}/N^{i-1}_{C/Y}.
\]

\section{Special Subvarieties in $X=\Gamma \backslash D$} \label{griffiths}

As far as we know, there is no good notion of Hecke operators on Mumford-Tate domains $D=G(\R)/V$. In addition, 
there are no good compactifications of $X=\Gamma \backslash D$ known in these cases, since 
$X$ does not even carry any algebraic structure in general \cite{grt}. Here $\Gamma$ is the monodromy group.

Therefore, to avoid these two difficulties, by a \emph{special curve} in $X$ we will denote an \'etale morphism
\[
\varphi^0: C^0 \longrightarrow X=\Gamma \backslash G(\R)/V                                                                    
\]
from a Shimura curve $C^0$ to $X$, induced from a morphism of algebraic groups $G' \to G$ defined over $\Q$.
In other words, $C^0$ is the quotient of the orbit of a certain Hodge structure $h \in D$ under the conjugation action of $G'$. 
One can generalize this notion to arbitrary special subvarieties. \\

Let $h\in D$ be a Hodge structure with Mumford-Tate group $M_h$. The $M_h(\mathbb{R})$-orbit of $h$ under the conjugation action of $M_h(\mathbb{R})$ in $D$ is a Mumford-Tate domain 
$D({M_h})\subset D$ in the sense of \cite{ggk,cmp}. In other words, $D(M_h)$ is a connnected component of 
the image of a Shimura datum $(M,X')$ in the Mumford-Tate datum $(G,X)$, see \cite[Chap. 17]{cmp}. Using this notation, we define: 

\begin{definition}[]\label{specialdef2}
A \emph{special subvariety of Shimura type} in $X$ is a horizontal, algebraic subvariety $Z^0 \subset X$, such that there is a Hodge structure $h \in D$ 
with Mumford-Tate group $M=M_h$ such that $Z^0$ is a quotient of the orbit $D(M)$ by a discrete subgroup.  
\end{definition}

Note that we shall omit reference to the Hodge structure $h$ and simply speak of a
Mumford-Tate domain $D(M)$ as in the above definition. 

Hence, by our definition of a special subvariety $Z^0$, we have a commutative diagram
\[
\begin{xy}
\xymatrix{
D(M) \ar[d]^{} \ar[r]^{}&   D \ar[d]^{} \\
Z^0 \ar[r]^{} &  X } 
\end{xy}
\]
Note that we always require a special subvariety to be horizontal and algebraic, so that $Z^0$ is of Shimura type, i.e., 
$D(M)$ is a Hermitian symmetric domain. In most cases, $Z^0$ is a proper subvariety of $X$ by \cite{grt}.

\begin{remark}
More general notions of special subvarieties in Mumford-Tate varieties are conceivable, for example horizontal subvarieties of maximal dimension 
in Mumford-Tate varieties. But it is not clear whether such definitions have good properties. For example such varieties may not carry any CM points. 
\end{remark}

\section{Relative Proportionality in $X=\Gamma \backslash D$} \label{relprop2}

To define the relative proportionality condition (RPC), using the notation of the previous paragraph, we need first the following
observations.

Let $C^0 {\buildrel \varphi^0 \over \longrightarrow} Y^0 {\buildrel i \over \hookrightarrow}X$ be a
special curve and $Y^0$ an algebraic subvariety of $X=\Gamma \backslash D$. Let $Y$ be a smooth compactification of $Y^0$ and $C$ a compatible 
smooth compactification of $C^0$, which extends to a finite morphism $\varphi: C \to Y$. Note that for this we do not need to 
require that $X$ has an algebraic compactification. 

Denote by $S_Y=Y \setminus Y^0$ the boundary divisor, and by $S_C=C \setminus C^0$, so that $S_C$ is the pullback of $S_Y$ to $C$. 

Fix a base point corresponding to a Hodge representation $h: \SSS \to G_\R$ whose orbit under $G$ defines $X$. 
Let ${\mathfrak g}={\rm Lie}(G)$. As mentioned earlier, ${\mathfrak g}=\bigoplus_p {\mathfrak g}^{-p,p}$  carries a weight zero Hodge structure. 
If $K$ is a maximal compact subgroup containing $V$, then its complexified Lie algebra ${\mathfrak k}_{\mathbb C}$ is given by the sum for even $p$, whereas 
its complement ${\mathfrak p}_{\mathbb C}$ is the sum for all odd $p$ \cite[Sec. 12.5]{cmp}. For $p=1$, we obtain 
the horizontal, holomorphic tangent bundle, see Definition~\ref{horizontal}. The vertical tangent bundle is given by the quotient of Lie algebras
${\mathfrak k}/{\mathfrak v}$, where ${\mathfrak v}$ is the Lie algebra of $V$. This terminology comes from the fibration \cite{cmp,grt}
\[
\omega: D=G(\R)/V \longrightarrow G(\R)/K.
\]

We remark that the horizontal tangent bundle agrees with the usual tangent bundle, if $V=K$ and $D$ is a Hermitian symmetric domain, for
example in the case of $X=A_g$.\\

\begin{proposition} Assume that $Y^0$ (and hence $C^0$) has unipotent monodromies at infinity. Then  
the bundle $(\varphi^0)^*T^h_X$ on $C^0$ extends to a vector bundle on $C$ which we denote by $\varphi^*T^h_X(- \log S)$. 
\end{proposition}

\begin{proof} Let $\mathcal{V}^{p,q}$ be the universal vector bundles on $D$ which parametrize the $(p,q)$-classes on $X$. 
The horizontal, holomorphic tangent bundle $T^h_X$ of $X$ is contained in a direct sum of the Hodge bundles: 
\[
T^h_{X}\subset {\mathcal End}^{-1,1} \left( \bigoplus_{p,q} \mathcal{V}^{p,q} \right) = 
\bigoplus_{p,q} {\mathcal Hom} \left( \mathcal{V}^{p,q},\mathcal{V}^{p-1,q+1}\right).
\]
All these bundles are homogenous on $D$, and the inclusion of the subbundle $T_{X}$ is defined by explicit conditions.
Over the algebraic variety $Y^0$, the restricted bundles $\mathcal{V}^{p,q}|_{Y^0} $ on the right hand side, and 
also the subbundle $T_{X}|_{Y^0}$, have a Deligne extension 
$\overline{\mathcal V}^{p,q}|_{Y^0}$ to $Y$. Therefore, $T^h_X|_{Y^0}$ and $(\varphi^0)T^h_X$ have natural extensions to $Y$ and $C$ which we denote by
$\varphi^*T^h_X(- \log S)$, although $S$ does not exist.
\end{proof}

Using this, we can define the logarithmic normal bundle $N_{C/X}$ through the exact sequence
\[
0 \to T_C(- \log S_C) \to \varphi^* T^h_X(- \log S) \to N_{C/X} \to 0.
\]
In a similar way, we have the exact sequence 
\[
0 \to T_C(- \log S_C) \to \varphi^* T_Y(- \log S_Y) \to N_{C/Y} \to 0.
\]
By a previous result \cite[Prop. 1.5.(ii)]{mz11} of ours, see Prop.~\ref{splitting}(ii) above, which is independent of $A_{g}$, 
we know that the logarithmic tangent bundle $T_{C}(-\log S_{C})$ is an orthogonal direct summand of the newly defined bundle $\varphi^*T_{X}(-\log S)$ 
with respect to the Hodge metric: 
\[
T_{C}(-\log S_{C})\hookrightarrow \varphi^{*}T_{X}(-\log S).
\]
We now show that certain local systems on $C^0$ split in a controlled way, giving a representation-theoretic proof of the following result of 
\cite{viehweg-zuo}. 

\begin{lemma} \label{decomp_lemma} Assume that $\mathbb{V}$ is a $\mathbb{C}$-variation of Hodge structures 
of weight $k$ over $C^0$ which comes from a $G(\R)$-representation on $X$ by restriction. Then, 
\[
\mathbb{V}=\mathbb{U} \oplus \bigoplus_i \left(S^{i}(\mathbb{L})\otimes
\mathbb{T}_{i} \right),
\]
where $\mathbb{L}$ is a weight one local system of rank $2$ and
$\mathbb{T}_{i}$ and $\mathbb{U}$ are unitary local systems of
weights $k-i$ and $k$ respectively.
\end{lemma} 

\begin{proof} Since $C^0$ splits in at least one place, we may
assume that the special Mumford-Tate group of the Shimura curve $C^0$ has
the form $SL(2) \times U_{1}\times...\times U_{r}$ for some $r \ge 0$,
where the $U_{i}$ are compact Lie groups (i.e., anisotropic). This gives
rise to an embedding $SL(2) \times
U_{1}\times...\times U_{r} \hookrightarrow G$ of algebraic groups. Now, since the groups
$U_{i}$ are compact as real groups, it follows that every
representation of them is a unitary representation and it is
well-known that the representations of the group $SL(2)_\R$ are
direct sums of symmetric products of the standard representation.
Note also that the irreducible subrepresentations of the product
representation is a product of the irreducible subrepresentations
of each representation and that the product of unitary
representations is again unitary. This means that there is a
standard $2$-dimensional representation $\mathbb{L}$ and unitary
representations $\mathbb{T}_{i}$ and $\mathbb{U}$ such that
$\mathbb{V}$ has the asserted decomposition.
\end{proof} 

Note that, since $\mathbb{L}$ is a weight $1$ variation of Hodge
structures, and $C^0$ is a special curve, by results of \cite{viehweg-zuo}, its Deligne extension to $C$
corresponds to a Higgs bundle of the form
$(\mathcal{L}\oplus\mathcal{L}^{-1}, \sigma)$ such that the Higgs
field $\sigma:\mathcal{L}\to \mathcal{L}^{-1}\otimes
\Omega^{1}_{C}(\log S_{C})$ is an isomorphism, and hence
$\mathcal{L}^{2}\simeq \Omega^{1}_{C}(\log S_{C})$.

We can now apply Lemma~ \ref{decomp_lemma} to the universal $\C$-local system $\V_\C$ of weight $k \ge 1$ on $X$ coming from the 
$\Q$-local system $_\V$ described above. It implies that $(\varphi^0)^*\mathbb{V}_\C= \mathbb{U} \oplus \bigoplus_i \left( S^{i}(\mathbb{L})\otimes \mathbb{T}_{i} \right)$
for local systems $\mathbb{L}$, $\mathbb{T}_{i}$ and $\mathbb{U}$ over $C^0$. We denote the Higgs bundles on $C$ 
corresponding to the $\C$-local systems $(\varphi^0)^*\V_{\C}$, $\mathbb{T}_{i}$ and $\mathbb{U}$ by $\mathcal{V}$, 
$\mathcal{T}_i$ and $\mathcal{U}$. The bundles $\mathcal{T}_i$ and $\mathcal{U}$ have degree $0$ and their Higgs fields are zero. 
Note that the Higgs field of $S^{i}(\mathbb{L})$ comes from that of
$\mathcal{L}$, i.e., is equal to $S^{i}(\sigma)$ for $\sigma$ the Higgs field of $\mathcal{L}$. The Higgs field of $\mathbb{V}$ 
respects the direct sums and vanishes on $\mathbb{U}$. Therefore, 
\[
T_{C}(-\log S_{C}) \subseteq \bigoplus_\Box
{\mathcal Hom}\left(\mathcal{L}^{i-2\mu}\otimes \mathcal{T}_{i,a}, \mathcal{L}^{j-2\nu}\otimes \mathcal{T}_{j,b}\right) 
\subset {\mathcal End}^{-1,1} \left( \bigoplus_{p+q=k} \mathcal{V}^{p,q} \right), 
\]
where the bundles $\mathcal{T}_{i,a}$, $\mathcal{T}_{j,b}$ have slope $0$, and 
$\Box=\{(\mu,i,\nu,j, a, b) \in \N_0^6 \mid \mu \le i \le k, \; \nu \le j \le k, \; a\le k-i, \; b\le k-j , \; j+b-\nu=i+a-\mu-1\}$. 
In the above sum, $T_{C}(-\log S_{C})$ is a direct summand and orthogonal with respect to the
natural Riemannian (i.e., Hodge) metric. Let $T_{C}(-\log S_{C})^{\perp}$ denote the orthogonal
complement of $T_{C}(-\log S_{C})$ in this sum. Thus, there is a decomposition
\[
\varphi^{*}T_{X}(-\log S)= T_{C}(-\log S_{C})\oplus N_{C/X}, 
\]
such that, as in \cite[Section 1]{mz11},
\[
N_{C/X} \subset T_{C}(-\log S_{C})^{\perp} \oplus \bigoplus_{p+q=k}{\mathcal Hom}
\left( \mathcal U^{p,q}, \mathcal V^{p-1,q+1} \right) \oplus \bigoplus_{p+q=k}{\mathcal Hom}
\left( \mathcal V^{p,q}, \mathcal{U}^{p-1,q+1} \right ).
\]
In particular, $N_{C/X}$ is a sum of polystable bundles of different slopes. Hence, one has a Harder-Narasimhan decomposition
\[
N_{C/X}=\bigoplus_{i=0}^s R_{i}
\]
with polystable bundles $R_i$ of strictly increasing slopes $\mu(R_{i}) < \mu(R_{i+1})$.
The length $s$ is an integer depending on $C$ and $X$.

Accordingly, the Harder-Narasimhan filtration on $N_{C/X}$ is given by 
\[
N^{i}_{C/X}=R_{0}\oplus
R_{1}\oplus \cdots \oplus R_{i}, \text{  } 0\leq i \leq s. 
\]
Taking the induced filtration $N^i_{C/Y}:=N^i_{C/X} \cap N_{C/Y}$ on $N_{C/Y}$ obtained
by intersection, we get a filtration on $N_{C/Y}$: 
\[
N^{0}_{C/Y}\subset N^{1}_{C/Y}\subset \cdots \subset
N^{s}_{C/Y}=N_{C/Y}.
\] 
In analogy with the $A_g$ case, we can now make the following definition:

\begin{definition}[Relative Proportionality Condition (RPC)]{~}\\ 
We say that $\varphi:C\to Y$ satisfies the \emph{relative proportionality condition} (RPC), if the
slope inequalities
\[
\mu(N^{i}_{C/Y}/ N^{i-1}_{C/Y})\leq \mu(N^{i}_{C/X}/
N^{i-1}_{C/X}), \text{   } i=0,\ldots,s
\]
are equalities. 
\end{definition}

Adding all these inequalities, we obtain a single inequality 
\[
\deg N_{C/Y} \leq r(C,Y,X) \cdot \deg T_{C}(-\log S_{C}), 
\]
where $r(C,Y,X) \in \Q$ is a rational number depending on $C$, $Y$ and $X$, and hence on $G$. 
However, it is not possible to write $r(C,Y,X)$ in a closed form, as in the case of $X=A_g$, since it would depend on $G$ and not only on the weight $k$.
In example~\ref{surfaces}, the constant $r(C,Y,X)$ is $1$ in the case of Hilbert modular surfaces and $\frac{1}{2}$ in the case of ball quotients.
The assertions of Proposition~\ref{splitting} and \cite[Formula 1.3]{mz11} also hold in this more general case by induction over $s$, i.e., we have a splitting 
\[
\varphi^*  T_Y(-\log S_Y) \cong T_C(-\log S_C) \oplus \bigoplus_{i=0}^s N^{i}_{C/Y}/N^{i-1}_{C/Y},
\]
if $C^0$ is a special curve in $X$ satisfying (RPC). 

\section{Proof of Theorem~\ref{Theorem3}}

In this section we prove Theorem~\ref{Theorem3}. From this, Theorem~\ref{Theorem1} and 
Theorem~\ref{Theorem2} follow, as we showed in the introduction.

We assume conditions (BIG) and (RPC) and look at a smooth and horizontal subvariety $Y^0 \hookrightarrow X$, 
where $X=\Gamma \backslash G(\R)/K$ is a (connected) Mumford-Tate variety. 

We choose a base point $y \in \bigcup_i C_i^0$. Note that $X$ carries a family of $\Q$-Hodge structures $\V$ as a local system.
It does not underly a variation of $\Q$-Hodge structures in general, since Griffiths transversality may not hold. 
However, when restricted to $Y^0$, or the curves $C_i^0$, this will be the case, since $Y^0$ is horizontal. We now consider 
the restriction of $\V$ to $Y^0$ only.

Let $H$ be the $\Q$-algebraic group from condition (LIE). By Definition~\ref{kernels}, it fixes precisely the Hodge classes in these vector spaces $W_{y \in Y}$.
For ease of notation, set $C=C_i$ in what follows (till end of the proof of Lemma~\ref{thickening}) and denote the restriction ${\V_{\C}}$ to $C$ by $\V_{C}$.
Since the $\C$-local system $\V_{C}$ has the form $\mathbb{V}_{C}=\bigoplus
\left( S^{i}(\mathbb{L})\otimes \mathbb{T}_{i} \right) \bigoplus \mathbb{U}$,
where $\mathbb{L}$ is related to a $\C$-local system of weight $1$
corresponding to a Higgs bundle
$\mathcal{L}\oplus\mathcal{L}^{-1}$, by Lemma~\ref{decomp_lemma} and (proof of) \cite{viehweg-zuo}, Proposition 3.4, the Higgs field is given by
\[
S^{i-2\mu}(\sigma): \mathcal{L}^{i-2\mu}\to
\mathcal{L}^{i-2\mu-2}\otimes \Omega^{1}_{Y}(\log S).
\]
In particular, the sheaves $E^{p,q}$ can be decomposed into a direct
sum of  polystable sheaves $E^{p,q}_{\iota}$ of slopes
$\mu(E^{p,q}_{\iota})=\iota \deg \mathcal{L}$ for $\iota \in [-qk,\cdots, pk]$.
Using this, we prove:

\begin{lemma} \label{thickening} The thickening $\theta_{y \in Y}$ on $E^{p,q}_{\iota}$
decomposes as a direct sum of morphisms:
\[
E^{p,q}_{\iota}\xrightarrow{\theta_{N^{i}_{C/Y}/N^{i-1}_{C/Y}}}E^{p-1,q+1}_{\iota+r_i}\otimes (N^{i}_{C/Y}/N^{i-1}_{C/Y})^{\vee}
\]
between polystable bundles of the same slope. Here, $r_i$ is the number satisfying $\mu(R_i)=\mu(N^{i}_{C/Y}/N^{i-1}_{C/Y})=r_i \deg \mathcal{L}$.
\end{lemma} 

\begin{proof} Note that the above decomposition of $N_{C/Y}$
gives a corresponding decomposition as
\[
\theta_{C}+\theta_{N_{C/Y}}=\theta_{C}+\theta_{N_{C/Y}^{0}}+\theta_{N_{C/Y}^{1}/N_{C/Y}^{0}}+ \cdots +\theta_{N_{C/Y}^{s}/N_{C/Y}^{s-1}}.
\]
Since, by Lemma~ \ref{decomp_lemma} on the curve $C$, one has the $\C$-local system
$\mathbb{V}_{C}=\bigoplus \left(  S^{i}(\mathbb{L})\otimes \mathbb{T}_{i} \right) \bigoplus \mathbb{U}$, the description of the sheaves
$E^{p,q}$ shows that we can reduce to the situation $i=1$.
This case is treated in \cite[Lemma 2.7]{mz11} for $\mathbb{V}_{C}^{\otimes k}$. In fact, if $i=1$, then for $k=1$ we have the decompositions
\[
\mathcal{L}\otimes \mathcal{T}\to \mathcal{L}^{-1}\otimes
\mathcal{T}\otimes \Omega^{1}_{C}(\log S_{C})
\]
\[
\mathcal{L}\otimes \mathcal{T}\to \mathcal{L}^{-1}\otimes
\mathcal{T}\otimes (N^{0}_{C/Y})^{\vee}
\]
\[
\mathcal{L}\otimes \mathcal{T}\to
\mathcal{U}^{\vee}\otimes(N^{1}_{C/Y}/N^{0}_{C/Y})^{\vee}
\]
\[
\mathcal{U}\to \mathcal{U}^{\vee}\otimes
(N^{2}_{C/Y}/N^{1}_{C/Y})^{\vee}
\]
and for arbitrary weight $k$, the result can be obtained by reducing to
the case $k=1$ by remembering that $\theta^{\otimes k}_{y \in Y}$ is
defined by the Leibniz rule.
\end{proof} 

Thus, we have shown that the kernels of $\vartheta$ decompose into vector bundles with vanishing slopes, and hence induce unitary Higgs bundles. 
This is the crucial ingredient for the remaining proof.  

\begin{proposition} Under condition (RPC), the subspaces $W_{y \in Y,\Q}$ and $W_{y \in Y}$ are invariant under the monodromy action of 
$\pi_1(\bigcup_i C_i^0,y) \to G$ and define a unitary local system on each curve $C_i^0$.
If conditions (RPC) and (BIG) both hold, then the subspaces $W_{y \in Y,\Q}$ and $W_{y \in Y}$ are invariant under the monodromy action 
of $\pi_1(Y^0,y) \to G$. 
\end{proposition}

\begin{proof} The proofs of Prop. 2.4 and Prop. 3.1 of \cite{mz11} immediately
carry over to this more general situation. Indeed the kernels $ker (\theta^{p,p})$ decompose to subbundles with non-positive
slopes and one shows that $ker (\theta^{p,p})_0$, the subbundle of slope 0, is invariant under complex conjugation induced by the real
structure and underlies a unipotent subsystem $\W$ of type $(p,p)$. Conversely, any $t\in W_{y\in Y, \R}$ lies in  
$ker (\theta^{p,p})_{0,y}$ showing that $\W_{\R,y}=W_{y\in Y, \R}$. Proposition 3.1 of \cite{mz11} says that the subspace
$W_{y\in Y, \R}$ is invariant under monodromy action. Fixing a base point $y_1\in C^{0}_1$ the idea is to study the parallel transport of real
vectors in $W_{y_1\in Y, \R}$ along paths in the connected subspace $\bigcup_i C_i$. The Higgs field $\theta_{y\in Y}$ decomposes as
$\theta_{y\in C_1}+\theta_{N_{C_1/Y},y}$ and hence $W_{y\in Y, \R}$ can be identified with the kernel of $\theta_{N_{C_1/Y},y}$.
On the other hand, as indicated above, there is a unitary subsystem $\W_{C_1}$ of type $(p,p)$ such that $W_{y\in Y, \R}=\W_{C_1,y}$.
The kernel $\theta_{N_{C_1/Y}}$ is a polystable subbundle of slope zero and hence it underlies a unitary subsystem $\W^{\prime}_{C_1}$. 
It follows that $W_{y\in Y, \R}\subset \W^{\prime}_{C_1,y}$. Now starting with a real vector $t_1\in W_{y\in Y, \R}$ and parallel transporting
it along some path in $C^{0}_1$ from $y_1$ to $y_2\in C^{0}_1\cap C^{0}_2$ to a vector $t_2$, one sees that $t_2$ is also a real vector and 
is contained in the fiber of $\W^{\prime}_{C_1}$ at $y_2$. It therefore follows that $\theta_{y\in C_1}(t_2)=0$ and $\theta_{N_{C_1/Y},y}(t_2)=0$, i.e.,
$t_2\in W_{y_2\in Y, \R}$. Regarding $y_2\in C^{0}_2$ and repeating the above parallel transport along a path in $C^{0}_2$, from 
$y_2$ to $y_3\in C_2\cap C_3$ and so on, one concludes that $W_{y\in Y, \R}$ is invariant under the monodromy action.\\
Finally, the last part of the proof there uses only condition (BIG).
\end{proof}

As in Cor. 3.5 of \cite{mz11}, one gets the following corollary: 

\begin{corollary}
The subspaces $W_{y \in Y}$ define a unitary local subsystem $\U \subset \W$ on $Y^0$ with $\Q$-structure. 
The local system $\U$ extends to $Y$, and has finite monodromy.
\end{corollary}

Since we assumed that the monodromies at infinity are unipotent, which always holds after a finite \'etale cover of $Y^0$, 
this means that $\U$ is trivial, and all its global sections, i.e., all $(p,p$)-classes inside $\U$, 
which are by definition of $W_{y \in Y}$ invariant under $H$, are also monodromy-invariant. 
Recall that the Mumford-Tate group $MT(\V_{Y^0})$ is the $\Q$-algebraic group 
fixing all parallel Hodge classes for all $p$ \cite[Chap. 15]{cmp}. We obtain therefore:

\begin{corollary}
The infinitesimally fixed Hodge classes in $\W_\Q$ over points $y \in Y^0$ are globally monodromy-invariant. 
Hence, condition (Mon) holds. 
\end{corollary}

Therefore, Theorem~\ref{Theorem3} is proven.

\end{document}